\documentclass[11pt]{article}
\usepackage[latin1]{inputenc}
\usepackage[T1]{fontenc}
\usepackage{amsmath,amssymb,amstext}
\usepackage{theorem}
\usepackage{a4wide}
\usepackage{graphicx,color}
\usepackage{multicol}
\usepackage[english]{babel}
\usepackage{enumerate}
\usepackage{eufrak}

\newtheorem{theorem}{Theorem}[section]

\newtheorem{example}[theorem]{Example}

\newtheorem{lemma}[theorem]{Lemma}

\newtheorem{proposition}[theorem]{Proposition}

\newenvironment{proof}[1][Proof]{\textbf{#1.} }{\ \rule{0.5em}{0.5em}}

\renewcommand{\geq}{\geqslant}
\def\leq{\leqslant}

\def\1{{\mathbf{1}}}

\def\1{{\mathbf{1}}}
\def\0.5{{\frac{1}{2}}}

\renewcommand{\thefootnote}{\fnsymbol{footnote}}
\begin{document}
\begin{center}
{\Large \textbf{Volatility Estimation of General Gaussian Ornstein-Uhlenbeck Process}} \\[0pt]
~\\[0pt]
Salwa Bajja\footnote{ National School of Applied Sciences -
Marrakesh, Cadi Ayyad University, Marrakesh, Morocco. Email:
\texttt{salwa.bajja@gmail.com }},
Qian Yu\footnote{ School of Statistics, East China Normal University, Shanghai, 200241, China.
 Email: \texttt{qyumath@163.com}}
 \\[0pt]
\end{center}

\begin{abstract}
\medskip In this article we study the asymptotic behaviour of the realized
quadratic variation of a process $\int_{0}^{t}u_{s}dG^{H}_{s}$,
where $u$ is a $\beta$-H\"older continuous process with $\beta >1-H$ and $G^H$ is a self-similar Gaussian process with parameters $H\in(0,3/4)$.
We prove almost sure convergence uniformly in time, and a stable weak convergence for the realized quadratic variation. As an application, we construct strongly consistent estimator for the integrated volatility parameter in a model driven by $G^H$.
\end{abstract}

\renewcommand{\thefootnote}{\arabic{footnote}} \renewcommand{\thefootnote}{%
\arabic{footnote}} 

\section{Introduction}

The realized quadratic variation is a powerful tool in the statistical analysis of stochastic processes, and it has received
a lot of attention in the literature. Furthermore, its generalization, the realized power variation of order $p>0$, have
received similar attention as it can tackle with several problems related to realized quadratic variation. For example, the
asymptotic normality does not hold for realized quadratic variation in the case of the fractional Brownian motion (fBm) $B^H$ with
$H>\frac34$, while asymptotic normality hold for realized power variation if one chooses $p$ large enough. Many results are limited
to the fBm who has the stationary increment, but not to the general non-stationary Gaussian process.

The realized power variation of order $p$ (quadratic variation if
$p=2$) is defined as
\begin{equation}  \label{eq:statistic}
\sum_{i=1}^{[nt]}\left|X_{i/n}-X_{(i-1)/n}\right|^p
\end{equation}
where $\{X_{t},t\geq0\}$ is a stochastic process. It was originally
introduced in Barndorff-Nielsen and Shephard (\cite{BS2002}, \cite{BS2003}, \cite{BS2004a},\cite{BS2004b}) to estimate the integrated
volatility in some stochastic volatility models used in quantitative finance and also, under an appropriate modification,
to estimate the jumps of the processes. The main interest in the mentioned papers is the asymptotic behaviour of appropriately
normalised version of the statistic (\ref{eq:statistic}), when the process $X_{t}$ is a stochastic integral with respect to a
Brownian motion. Refinements of the results have been obtained in \cite{W2003} and \cite{W2005}, and further extensions can be found
in \cite{BSal}.

The asymptotic behaviour of the power variation of a stochastic integral $Z_{t}=\int_{0}^{t}u_{s}dB_{s}^{H}$ with respect to a fBm was studied in \cite{CNW}. In \cite{CNW} the authors proved that if $u=\{u_{t},t\geq0\}$ has finite $q$-variation for some $q<1/(1-H)$,
then
\begin{eqnarray}
n^{-1+pH}V_{p}^{n}(Z)_{t} &\longrightarrow&
c_{1,p}\int_{0}^{t}|u_{s}|^{p}ds
\end{eqnarray}
uniformly in probability in any compact sets of t, where $c_{1,p}=\mathrm{I%
\kern-0.16em E}[|B_{1}^{H}|^{p}]$ and  $V_{p}^{n}(Z)_{t}=\sum_{i=1}^{[nt]}\left|Z_{i/n}-Z_{(i-1)/n}\right|^p.$ The authors also proved central
limit
theorem for $H \in (0,\frac{3}{4}]$. However, the condition $H\in(0,\frac{3}{%
4}]$ is critical in \cite{CNW}. The first objective of \cite{HNZ}
was to remove this restriction. They used higher order differences
and defined the
power variation as $V_{k,p}^{n}(Z)_{t}=\sum_{i=1}^{[nt]-k+1}\left|%
\sum_{j=0}^{k}(-1)^{k-j}C_{j}^{k}Z_{(i+j-1)/n}\right|^{p}$ for
certain numbers $C_j^k$.

On a related literature we mention also a series of articles, all by the same authors, studying power variations of
general Gaussian processes. In \cite{BCP2009} asymptotic theory for the realized power variation of the processes $\phi(G)$ was studied.
Here $G$ is a general Gaussian process with stationary increments, and $\phi$ is a deterministic function. The authors proved that
under some mild assumptions on the variance function of the increments of $G$ and certain regularity conditions on the path of
the process, a properly normalised converge uniformly in probability. Exploiting these ideas, central limit theorems and
convergence of (multi) power variations for the general Gaussian processes with stationary increments and Gaussian semistationary
processes was studied in \cite{BCPW2009} and \cite{BCP2011}. Finally, similar questions for variations based on higher order
differences were studied in \cite{BCP2013}. As an application, estimation of the smoothness parameter of the process was discussed.

While the literature on the topic is wide due to the centrality of the problem,
all of the mentioned studies consider only (uniform) convergence in probability.
To the best of our knowledge, stronger mode of convergence such as uniform almost sure convergence is not widely studied in the literature.
In the paper \cite{BEV}, they studied the asymptotic behaviour of
the
realized quadratic variation of a process of the form $\int_{0}^{t}u_{s}dY^{(1)}_{s}$, where $Y_{t}^{(1)}=\int_{0}^{t}e^{-s}dB_{a_{s}}$; $a_{t}=He^{t/H}$, $B^{H}$ is a fBm with Hurst parameter $H\in(0,1)$, and $u$ is a $\beta$-H\"older continuous process with $\beta > 1-H$. such that the process $Y^{(1)}$ is connected to the fractional Ornstein-Uhlenbeck process of the second kind, that is defined through the Lamperti transform
of the fBm. Equivalently, fractional Ornstein-Uhlenbeck process of the second kind can be defined as
the solution to the stochastic differential equation \begin{eqnarray}
dX_{t} &=& -\theta X_{t}dt+\sigma_{t} dY^{(1)}_{t}.
\end{eqnarray}
As the main result, they obtained the almost sure and uniform
convergence. In comparison, \cite{CNW} obtained uniform
convergence in probability. They
also established weak convergence result provided that $H\in\left(0,\frac34%
\right)$.

In this paper we study the asymptotic behaviour of the realized
quadratic variation of a process of the form $\int_{0}^{t}u_{s}dG^{H}_{s}$, where $%
G^{H}$ is a self-similar Gaussian process (including fBm $B^H$, sub-fBm $S^H$ and bi-fBm $B^{H_0,K_0}$) with parameter $H\in(0,3/4)$ ($H=H_0K_0$ for bi-fBm) and $u$ is a $\beta$-H\"{o}lder continuous process with $\beta > 1-H$.  The Guaussian Ornstein-Uhlenbeck process can be
defined as the solution to the stochastic differential equation
\begin{eqnarray}
dX_{t} &=& -\theta X_{t}dt+\sigma_{t} dG^{H}_{t}.
\end{eqnarray}

As our main result, we obtain almost sure and uniform convergence of the realized quadratic variation of the self-similar Gaussian process $G^{H}$. That is, we  show that for $Z_t = \int_{0}^{t}u_{s}dG^{H}_{s}$ we have
$$
\sum_{i=1}^{[nt]}\left|Z_{\frac{i}{n}}-Z_{\frac{i-1}{n}%
}\right|^2 \longrightarrow \int_{0}^{t}|u_{s}|^2ds
$$
almost surely and uniformly in $t$, for any $H\in(0,3/4)$ and any process $u$ that is regular enough. In order to obtain this stronger convergence, we apply recently developed simplified method \cite{Lauri Viitasaari} to study quadratic variations of Gaussian sequence. With this simplified method that is based
on a concentration phenomena, one is able to obtain stronger convergence at the same time.

To obtain the desired results, we make the following assumptions on the self-similar Gaussian process $G^H$:

\textbf{(A1)} Let $d(s,t)=\mathbb{E}(G_t^H-G_s^H)^2$ is in $C^{1,1}$ outside diagonal, which satisfies
$$|\partial_{s,t}d(s,t)|=O(|t-s|^{2H-2}).$$

\textbf{(A2)} $G^H$ is H\"{o}lder continuous of  order $\delta$ for any $0<\delta<H$.

\textbf{(A3)} Let $I_n(i)=\{j: ~\frac{j}{m}\in(\frac{i-1}{m},\frac{i}{m}]\}$. As $m\to\infty$,
$$m^{-1+2H}\sum_{j\in I_n(i)}\mathbb{E}|G^H_{j/m}-G^H_{(j-1)/m}|^2\to\frac1n.$$

\textbf{(A4)} For $j, l=1,2,\cdots, N$, there exist constants $c_0$ and $c_1$ such that
$$\mathbb{E}[(G^H_{j}-G^H_{j-1})(G^H_{l}-G^H_{l-1})]=c_0\rho_H(|j-l|)+c_1\theta(j,l)$$
where $\rho_H(x)=\frac12\Big[(x+1)^{2H}+(x-1)^{2H}-2x^{2H}\Big]$ and $|\theta(j,l)|^2=o(1/j)$ as $j\to\infty$ (or equal to $o(1/l)$ as $l\to\infty$).\\

Note that, assumptions (A1)--(A3) mainly used in the proof of consistency in Theorem \ref{theorem:Consistency}. Condition of $\theta(j,l)$ in (A4) such that for $m\geq2$, $$\lim_{N\to\infty}\frac1N\sum_{j,l}\Big|\mathbb{E}[(G^H_{j}-G^H_{j-1})(G^H_{l}-G^H_{l-1})]\Big|^m<\infty$$ which will given in the main proof of stable convergence in Theorem \ref{theorem:distribution asymptotic}.

The paper is outlined in the following way. After some preliminaries in Section 2, Section 3 is devoted to the proof of main results, based on the assumptions (A1)--(A4) in Section 1 and the Lemmas and Theorems given in Section 2. We apply our results to the estimation of the integrated volatility in Section 4.

Throughout this paper, if not mentioned otherwise, the letter $c$, with or without a subscript, denotes a generic positive finite constant and may change
from line to line.
\section{Preliminaries}

In this paper, we will consider  $\{G_t^H, t\geq0\}$ is a centered Gaussian process defined on some probability space $(\Omega, \mathcal{F}, P)$ with self-similar index $H\in(0,3/4)$.  We always assume that $G^H$ satisfies assumptions (A1)--(A4). This  conditions that are satisfied by a variety of Gaussian processes. In particular, it is straightforward to validate the following Gaussian processes.

\begin{example}\label{ex-fBm}
$G^H_t=B_t^H$ is a fBm, of which the covariance function is
$$\mathbb{E}(B_t^HB_s^H)=\frac12(t^{2H}+s^{2H}-|t-s|^{2H}).$$
\end{example}

\begin{example}\label{ex-sub-fBm}
$G^H_t=S_t^H$ is a sub-fBm, of which the covariance function is
$$\mathbb{E}(S_t^HS_s^H)=t^{2H}+s^{2H}-\frac12[(t+s)^{2H}+|t-s|^{2H}].$$
\end{example}

\begin{example}\label{ex-bi-fBm}
$G^H_t=B_t^{H_0,K_0}$ is a bi-fBm, of which the covariance function is
$$\mathbb{E}(B_t^{H_0,K_0}B_s^{H_0,K_0})=2^{-K_0}[(t^{2H_0}+s^{2H_0})^{K_0}-|t-s|^{2H_0K_0}],$$
\end{example}
where $H=H_0K_0\in(0,3/4)$ and $K_0\in(0,1]$.\\

Next, we are going to verify that these processes meet the assumptions (A1)--(A4).

\begin{lemma}\label{lem-fBm}
Assumptions (A1)--(A4) are satisfied by fBm.
\end{lemma}
\begin{proof}
For $t,s >0$,
$$d(s,t)=|t-s|^{2H}$$
which gives (A1) and (A2).

Since the fBm has the incremental stationarity, then

$$\mathbb{E}[(B^H_{j}-B^H_{j-1})(B^H_{l}-B^H_{l-1})]=\rho_H(|j-l|),$$
where $\rho_H(x)=\frac12[|x+1|^{2H}+|x-1|^{2H}-2|x|^{2H}]$.  This gives (A4).

For (A3),
\begin{align*}
m^{-1+2H}\sum_{j\in I_n(i)}\mathbb{E}|G^H_{j/m}-G^H_{(j-1)/m}|^2&=m^{-1}\sum_{j\in I_n(i)}\mathbb{E}[(B^H_{j}-B^H_{j-1})^2\\
&=m^{-1}\sum_{j\in I_n(i)}\rho_H(0)\\
&=\frac{i}{n}-\frac{i-1}{n}=\frac1n.
\end{align*}
This completes proof.
\end{proof}

\begin{lemma}\label{lem-sub-fBm}
Assumptions (A1)--(A4) are satisfied by sub-fBm.
\end{lemma}
\begin{proof}
For $t,s >0$, by Proposition 1.15 in Tudor \cite{Tudor}, we can see
$$(2-2^{2H-1})|t-s|^{2H}\leq d(s,t)\leq |t-s|^{2H}, ~~H>1/2$$
and
$$|t-s|^{2H}\leq d(s,t)\leq (2-2^{2H-1})|t-s|^{2H}, ~~H<1/2,$$
which gives (A1) and (A2).

By simple calculation, we can find
$$\mathbb{E}[(S^H_j-S^H_{j-1})(S^H_l-S^H_{l-1})]=\rho_H(|j-l|)-\rho_H(j+l-1).$$

It is easy to see that (A3) follows from the proof of Lemma \ref{lem-fBm} and
$$\sum_{j\in I_n(i)}\rho_H(2j-1)=\sum_{j\in(\frac{i-1}n,\frac{i}n]}\rho_{H}(2mj-1)=O(m^{2H-2}) ~\text{as} ~m\to\infty.$$

Since $\rho_H(n)$  is a monotonically decreasing function and is greater than zero when $H>1/2$, and  $\rho_H(n)$ is increasing and is less than zero for $H<1/2$, we have
$$\Big|\mathbb{E}[(S^H_j-S^H_{j-1})(S^H_l-S^H_{l-1})]\Big|\leq \Big|\rho_H(|j-l|)\Big|.$$
 Moreover, $|\rho_H(j+l-1)|^2=o(1/j)$ as $j\to\infty$ (or equal to $o(1/l)$ as $l\to\infty$) for $H<3/4$. This completes the proof.
\end{proof}

\begin{lemma}\label{lem-bi-fBm}
Assumptions (A1)--(A4) are satisfied by bi-fBm.
\end{lemma}
\begin{proof}
For $t,s >0$, by Proposition 1.7 in Tudor \cite{Tudor}, we can see
$$2^{-K_0}|t-s|^{2H}\leq d(s,t)\leq 2^{2-K_0}|t-s|^{2H}, $$
which gives (A1) and (A2).

Similar to sub-fBm, we have
$$\mathbb{E}[(B^{H_0,K_0}_j-B^{H_0,K_0}_{j-1})(B^{H_0,K_0}_l-B^{H_0,K_0}_{l-1})]=2^{1-K_0}\rho_{H_0K_0}(|j-l|)+\theta(j,l),$$
where
\begin{align*}
\theta(j,l)&=2^{-K_0}\Big[((j-1)^{2H_0}+l^{2H_0})^{K_)}+(j^{2H_0}+(l-1)^{2H_0})^{K_0}\\
&\qquad\qquad-(j^{2H_0}+l^{2H_0})^{K_0}-((j-1)^{2H_0}+(l-1)^{2H_0})^{K_0}\Big].
\end{align*}

By the Lemma 1.1 and the proof of Proposition 1.10 in Tudor \cite{Tudor}, we can have
$$2^{K_0}|\theta(j,j)|=|h(j)+2|,$$
where $h(x)=x^{2H_0K_0}+(x-1)^{2H_0K_0}-2^{1-K_0}(x^{2H_0}-(x-1)^{2H_0})^{K_0}$.

Thus, (A3) follows from
$$\sum_{j\in I_n(i)}\theta(j,j)=\sum_{j\in(\frac{i-1}n,\frac{i}n]}\theta(mj,mj)$$
and $h(mj)$ converges to zero, as $m\to\infty$.

Let $f_j(x)=(j^{2H_0}+x^{2H_0})^{K_0}-((j-1)^{2H_0}+x^{2H_0})^{K_0}>0$, which is decreasing with respect to $x$. Then  we can see
\begin{align*}
|\theta(j,l)|&=2^{-K_0}(f_j(l-1)-f_j(l))\leq2^{-K_0}f_j(0)\\
&=2^{-K_0}\left(j^{2H_0K_0}-(j-1)^{2H_0K_0}\right)\\
&\leq 2^{-K_0}|\rho_{H_0K_0}(j-1)|.
\end{align*}

When $H=H_0K_0<3/4$,
$$|\rho_{H}(j-1)|^m=o(1/j), ~\text{as} ~j\to\infty, ~~\text{for} ~m\geq2.$$
Similarly, we can obtain that
\begin{align*}
|\theta(j,l)|^m=o(1/l), ~~\text{as} ~l\to\infty, ~~\text{for} ~m\geq2.
\end{align*}

 This gives (A4).
\end{proof}\\

We refer to \cite{LN}, \cite{RT} and \cite{Tudor} for more details on sub-fBm and bi-fBm.

We also recall that, for $p>0$, the $p$-variation of a real-valued function $f$ on an interval [a,b] is defined as
\begin{equation}  \label{p-variation}
var_{p}(f;[a,b])=\sup_{\pi}\left(\sum_{i=1}^{n}|f(t_{i})-f(t_{i-1})|^p%
\right)^{1/p},
\end{equation}
where the supremum is taken over all partitions $\pi=\{a=t_{0}<t_{1}<...<t_{n}=b\}.$ We say that $f$ has finite
$p$-variation (over the interval $[a,b]$), if $var_p(f;[a,b])<\infty$. Young proved that the integral $\int_{a}^{b}fdg$ exists
as a Riemann-Stieltjes integral provided that $f$ and $g$ have finite $p$-variation and $q$-variation with $1/p+1/q>1$. Moreover, the following inequality holds:
\begin{equation}  \label{Young inequality}
\left|\int_{a}^{b}fdg-f(a)(g(b)-g(a))\right|\leqslant c_{p,q}var_{p}(f;[a,b])var_{q}(g;[a,b]),
\end{equation}
where $c_{p,q}=\zeta(1/q+1/p)$ , with $\zeta(s)=\sum_{n\geq1}n^{-s}.$

We denote by
\begin{equation*}
\parallel f\parallel_{\alpha}:=\sup_{a\leqslant s<t\leqslant b}\frac{|f(t)-f(s)|}{|t-s|^\alpha}
\end{equation*}
the H\"older seminorm of order $\alpha$. Clearly, if $f$ is $\alpha$-H\"older continuous, then it has finite $(1/\alpha)$-variation on
any finite interval. In this case we have, for any $p\geq \frac{1}{\alpha}$, that
\begin{equation}  \label{eq:holder_var}
var_{p}(f;[a,b]) \leqslant \parallel
f\parallel_{\alpha}(b-a)^\alpha.
\end{equation}
Throughout the paper, we also assume that $T< \infty$ is fixed. That is, we
consider stochastic processes on some compact interval. We denote by $\|.\|_{\infty}$ the supremum norm on $[0,T]$.\\

For any natural number $n\geq1$, and for any stochastic process $Z=\{Z_{t},t\geq0\}$, we write
\begin{eqnarray}  \label{quadratic variation of integral}
V_{n}(Z)_{t}=\sum_{i=1}^{[nt]}\left|Z_{\frac{i}{n}}-Z_{\frac{i-1}{n}}\right|^2.
\end{eqnarray}

We will use the following two general results, taken from \cite{Lauri Viitasaari}, on the convergence of the quadratic variations of a Gaussian
process.
\begin{theorem}
\cite[Theorem 3.1]{Lauri Viitasaari}) \label{theorem:BE_bound_QV}
Let $X$ be a continuous Gaussian process and denote by $V_n^X$ its quadratic variation defined by
\begin{equation*}
V_n^X = \sum_{k=1}^n \left[\left(\Delta_k X\right)^2 -\mathrm{I\kern-0.16em E}\left(\Delta_k X\right)^2 \right],
\end{equation*}
where $\Delta_k X = X_{t_{k}} - X_{t_{k-1}}$. Assume that
\begin{eqnarray}
\max_{1\leqslant j \leqslant N(\pi_{n})-1}\sum_{k=1}^{N(\pi_n)-1}\frac{1}{\sqrt{\phi(\Delta t_{k})\phi(\Delta t_{j})}}|\mathrm{I\kern-0.16em E}[(X_{t_{k}}-X_{t_{k-1}})(X_{t_{j}}-X_{t_{j-1}})]| &\leqslant& h(|\pi_{n}|) \notag
\end{eqnarray}
for some function $\phi$ and $h(|\pi_{n}|)$.\newline If $h(|\pi_{n}|)\rightarrow 0 $ as $|\pi_{n}|$ tends to zero, then
the convergence
\begin{eqnarray}
\left|\sum_{k=1}^{N(\pi_{n})-1}\frac{(X_{t_{k}}-X_{t_{k-1}})^2}{\phi(t_{k}-t_{k-1})}-\sum_{k=1}^{N(\pi_{n})-1}\frac{\mathrm{I\kern-0.16em E}%
(X_{t_{k}}-X_{t_{k-1}})^2}{\phi(t_{k}-t_{k-1})}\right|
&\rightarrow& 0
\end{eqnarray}
holds in Probability. Furthermore, the convergence holds almost surely provided that $h(|\pi_{n}|)=o(\frac{1}{log(n)}).$
\end{theorem}

The following lemma gives easy way to compute the function $h(n)$ and is essentially taken from \cite{Lauri Viitasaari} (see
\cite[Theorem 3.3]{Lauri Viitasaari}).

\begin{lemma}
\label{lemma:lemma of Lauri} \label{lma:rate}\cite{Lauri
Viitasaari} Let $X$
be a continuous Gaussian process such that the function $%
d(s,t)=E(X_{t}-X_{s})^{2}$ is in $C^{1,1}$ outside diagonal.
Furthermore, assume that
\begin{equation}\label{eq:assuption}
|\partial_{st} d(s,t)|=O\left(|t-s|^{2H-2}\right)
\end{equation}
for some $H\in (0,1), H \neq \frac12$. Then
\begin{equation*}
\max_{1\leqslant j \leqslant n} \sum_{k=1}^n
\left\vert\mathrm{I\kern-0.16em E}(\Delta_k X \Delta_j
X)\right\vert \leqslant \max_{1\leqslant j \leqslant
n} d\left(\frac{j}{n},\frac{j-1}{n}\right)+ \left(\frac{1}{n}%
\right)^{1\wedge 2H}.
\end{equation*}
\end{lemma}

Finally, in order to study stable convergence in law we recall the following general convergence result taken from \cite{CNP}.

\begin{theorem}[\protect\cite{CNP}]
\label{theorem: Hypotheses} Let $(\Omega,\mathcal{F},P)$ be a complete probability space. Fix a time interval $[0,T]$ and
consider a double sequence of random variables $\xi=\{\xi_{i,m},m\in Z_{+},1\leqslant i \leqslant [mT]\}.$ Assume
the double sequence $\xi$ satisfies the following hypotheses.\\
\textbf{(H1)} Denote $g_{m}(t):=\sum_{i=1}^{[mt]}\xi_{i,m}$. The finite dimensional distributions of the sequence of processes
$\{g_{m}(t),t \in [0,T]\}$ converges $\mathcal{F}$-stably to those of $\{B(t), t\in [0,T]\}$ as $m\rightarrow\infty$, where $\{B(t), t\in [0,T]\}$ is a standard
Brownian motion independent of $\mathcal{F}$.\\
\textbf{(H2)} $\xi$ satisfies the tightness condition $\mathrm{I\kern-0.16em E}\left|\sum_{i=j+1}^{k}\xi_{i,m}\right|^4\leqslant C \left(\frac{k-j}{m}\right)^2$ for any $1\leqslant j \leqslant k\leqslant [mT]$.

If $\{f(t), t\in [0,T]\}$ is an $\alpha$-H\^{o}lder continuous process with $\alpha>1/2$ and we set $X_{m}(t):=\sum_{i=1}^{[mt]}f(\frac{i}{m})\xi_{i,m},$ then we have the $\mathcal{F}$-stable convergence
\begin{eqnarray}
X_{m}(t)&\underset{m\rightarrow\infty}{\overset{Law}{\longrightarrow}}&\int_{0}^{t}f(s)dB_{s},  \notag
\end{eqnarray}
in the Skorohod space $\mathcal{D}[0,T]$ equipped with the uniform topology.
\end{theorem}

Recall that a sequence of random vectors or processes $Y_{n}$
converges $\mathcal{F}$-stably in law to a random vector or
process $Y$, where $Y$ is defined on an extention
$(\Omega',\mathcal{F}',P')$ of the original probability
$(\Omega,\mathcal{F},P)$, if
$(Y_{n},Z)\overset{Law}{\longrightarrow} (Y,Z)$ for any
$\mathcal{F}$-measurable random variable Z. If Y is
$\mathcal{F}$-measurable, then we have convergence in probability.
We refer to \cite{AE}, \cite{LL} and \cite{Renyi} for more details on stable
convergence.\\

At last of this section, we will give a useful lemma to prove the stable convergence by (A4).

\begin{lemma}\label{lem-fdd}
 Let $(a_k,b_k]$, $k=1,\cdots, N$ be pairwise disjoint intervals contained in $[0,T]$. Define
 $$G_k^{(n)}=n^{-H}\sum_{[na_k]<j\leq [nb_k]}(G_j^{H}-G_{j-1}^H)$$
 and
 $$Y_k^{(n)}=\frac1{\sqrt{n}}\sum_{[na_k]<j\leq [nb_k]}H_2(G_j^{H}-G_{j-1}^H)$$
for $k=1,\cdots, N$, where $H_2(x)=x^2-1$ is the $2$-th Hermite polynomial. Assume $H<3/4$ and $G$ satisfies (A1)--(A4), then we have
$$(G^{(n)}, Y^{(n)})\overset{\mathcal{L}}{\to}(G,V),$$
where $G$ and $V$ are independent centred Gaussian vectors, with $G_k=G^H_{b_k}-G^H_{a_k}$, and the components of $V$ are independent with variances $v_1^2(b_k-a_k)$ and $v_1$ is dependent on functions $\rho_H$ and $\theta$.
\end{lemma}

\begin{proof}
Denote by $\mathcal{H}_m$ the $m$-th Wiener chaos, the closed subspace of $L^2(\Omega, \mathcal{F}, P)$ generated by the random variables $H_m(X)$, where $X$ belongs to first Wiener chaos, $\mathbb{E}X^2=1$ and $H_m$ is the $m$-th Hermite polynomial. The mapping $I_m: ~\mathcal{H}_1^{\odot m}\to \mathcal{H}_m$ denoted by $I_m(X^{\otimes m})=H_m(X)$ is a linear isometry between the symmetric tensor product $\mathcal{H}_1^{\odot m}$, equipped with the norm $\sqrt{m!}||\cdot||_{\mathcal{H}_1^{\otimes m}}$. For function
$$H(X)=\sum_{m=2}^\infty c_mH_m(X)$$
with $\sum_{m=2}^\infty c_m^2m!=\mathbb{E}|H(Z)|^2<\infty$, $Z$ being an $N(0,1)$ random variable and
$$J_mH(X)=c_mH_m(X)$$
where $J_m$ denote the projection operator on the $m$-th Wiener chaos. Using the same ways as the proof of Proposition 10 in Corcuera, Nualart and Woerner \cite{CNW}, to prove the desired result, we only need to prove, for any $m\geq2$, $k=1,\cdots, N$,
\begin{equation}\label{eq-fdd-1}
\lim_{n\to\infty}\mathbb{E}|J_m\widetilde{Y}_k^{(n)}|^2=:\sigma^2_{m,k}<\infty,
\end{equation}

\begin{equation}\label{eq-fdd-2}
\sum_{m=2}^\infty\sup_{n}\mathbb{E}|J_m\widetilde{Y}_k^{(n)}|^2<\infty,
\end{equation}

\begin{equation}\label{eq-fdd-3}
\lim_{n\to\infty}\mathbb{E}[J_m\widetilde{Y}_k^{(n)}J_m\widetilde{Y}_h^{(n)}]=0, ~~k\neq h,
\end{equation}
and
\begin{equation}\label{eq-fdd-4}
\lim_{n\to\infty}I_m^{-1}J_m\widetilde{Y}_k^{(n)}\otimes_p I_m^{-1}J_m\widetilde{Y}_k^{(n)}=0, ~~1\leq p\leq m-1,
\end{equation}
where $$\widetilde{Y}_k^{(n)}=\frac1{\sqrt{n}}\sum_{[na_k]<j\leq [nb_k]}H(G_j^{H}-G_{j-1}^H).$$

Replace $\rho_H(|j-l|)$ by $\rho_H(|j-l|)+\theta(j,l)$, then  it is easy to obtain \eqref{eq-fdd-3} and \eqref{eq-fdd-4}, since $|\theta(j,l)|^2=o(1/j)$ as $j\to\infty$. So, we only need prove \eqref{eq-fdd-1} and \eqref{eq-fdd-2} below.

\begin{align*}
\mathbb{E}|J_m\widetilde{Y}_k^{(n)}|^2&=\frac{m!c_m^2}{n}\sum_{[na_k]<j,l\leq [nb_k]}\Big[\mathbb{E}(G^H_j-G^H_{j-1})(G^H_l-G^H_{l-1})\Big]^m\\
&=\frac{m!c_m^2}{n}\sum_{[na_k]<j,l\leq [nb_k]}\Big[c_0\rho_H(|j-l|)+c_1\theta(j,l)\Big]^m\\
&=\frac{m!c_m^2}{n}\sum_{[na_k]<j\leq [nb_k]}\Big[c_0\rho_H(0)+c_1\theta(j,j)\Big]^m\\
&\qquad+2\frac{m!c_m^2}{n}\sum_{[na_k]<j\neq l\leq [nb_k]}\Big[c_0\rho_H(|j-l|)+c_1\theta(j,l)\Big]^m.
\end{align*}

By assumption (A4), we can see the summation above with respect to  $\theta(j,l)$  part is finite, denoted by
\begin{align*}
\sigma^2_\theta:=\lim_{n\to\infty}\left(\frac{m!c_m^2}{n}\sum_{[na_k]<j\leq [nb_k]}\Big[c_1\theta(j,j)\Big]^m+2\frac{m!c_m^2}{n}\sum_{[na_k]<j\neq l\leq [nb_k]}\Big[c_1\theta(j,l)\Big]^m\right).
\end{align*}

Then \eqref{eq-fdd-1} and \eqref{eq-fdd-2} follow by
\begin{align*}
&\frac1n\sum_{[na_k]<j\leq [nb_k]}\rho_H(0)^m+\frac1n\sum_{[na_k]<j\neq l\leq [nb_k]}(\rho_H(|j-l|))^m\\
&\qquad=\frac{[nb_k]-[na_k]}{n}\rho_H(0)^m+\sum_{j=1}^{[nb_k]-[na_k]}\rho_H(j)^m\frac{[nb_k]-[na_k]-j}{n}\\
&\qquad\to(b_k-a_k)\rho_H(0)^m+\sum_{j=1}^\infty\rho_H(j)^m=:\sigma^2_\rho, ~~n\to\infty,
\end{align*}
and we denoted by $\sigma^2_{m,k}:=\lim_{n\to\infty}\mathbb{E}|J_m\widetilde{Y}_k^{(n)}|^2$
(since this is a complex binomial expansion related to $\rho_H$ and $\theta$, the calculation process of $\lim_{n\to\infty}\mathbb{E}|J_m\widetilde{Y}_k^{(n)}|^2$ is complicated, so we can only denote it by $\sigma^2_{m,k}$).

When $m=2$, we can compute the variance of the limit $\lim_{n\to\infty}\mathbb{E}|Y_k^{(n)}|^2=:v_1^2(b_k-a_k)$ with $v_1^2(b_k-a_k)c_2^2=\sigma^2_{2,k}$.
\end{proof}

\section{Main results}
We study the asymptotic behavior of the realized quadratic variation of a stochastic process of the form $\int_{0}^{t}u_{s}dG^{H}_{s}$,
where $u$ is a H\"older continuous process of order $\beta >1-H$. Note that, as $G^{H}$ is H\"older continuous of order
$H-\varepsilon$ by assumption (A2), the integral can be understood as a Riemann-Stieltjes integral. In particular, the process is well-defined.

We are now ready to state our first main result that provides us
the uniform strong consistency.

\begin{theorem}
\label{theorem:Consistency} Under the assumptions (A1)--(A3), we further suppose that $u=\{u_{t},t\in[0,T]\}$
is an H\"older continuous stochastic process of order $\beta$ with $\beta > 1-H$, $0<H<3/4$, and set

\begin{eqnarray}  \label{eq:process_Z}
Z_{t} &=& \int_{0}^{t}u_{s}dG^{H}_{s}.
\end{eqnarray}
Then, as $n$ tends to infinity,
\begin{eqnarray}
n^{2H-1}V_{n}(Z)_{t} &\longrightarrow & \int_{0}^{t}|u_{s}|^2ds,
\end{eqnarray}
almost surely and uniformly in $t$.
\end{theorem}

\begin{proof}
For $t\in [0,T]$ and an integer $n$, we denote by $[nt]$ the largest integer that is at most $nt$. Let now $m\geq n$. We have
\begin{eqnarray}
	m^{-1+2H}V_{m}(Z)_{t}&-&\int_{0}^{t}|u_{s}|^2ds  \notag \\
	&=&m^{2H-1}\sum_{j=1}^{[mt]}\left(\left|%
	\int_{(j-1)/m}^{j/m}u_{s}dG^{H}_{s}\right|^{2}-\left|u_{\frac{j-1}{m}%
	}(G^{H}_{\frac{j}{m}}-G^{H}_{\frac{j-1}{m}})\right|^2\right)  \notag \\
	&&+m^{2H-1}\left(\sum_{j=1}^{[mt]}\left|u_{\frac{j-1}{m}}\left(G^{H}_{%
		\frac{j}{m}}-G^{H}_{\frac{j-1}{m}}\right)\right|^2-\sum_{i=1}^{[nt]}%
	\left|u_{\frac{i-1}{n}}\right|^2\sum_{j\in I_{n}(i)}\left|G^{H}_{\frac{j}{m%
	}}-G^{H}_{\frac{j-1}{m}}\right|^2\right)  \notag \\
	&&+m^{2H-1}\sum_{i=1}^{[nt]}\left|u_{\frac{i-1}{n}}\right|^2\sum_{j\in
		I_{n}(i)}\left|G^{H}_{\frac{j}{m}}-G^{H}_{\frac{j-1}{m}%
	}\right|^2-n^{-1}\sum_{i=1}^{[nt]}\left|u_{\frac{i-1}{n}}\right|^2
	\notag \\
	&&+\left(n^{-1}\sum_{i=1}^{[nt]}\left|u_{\frac{i-1}{n}%
	}\right|^2-\int_{0}^{t}|u_{s}|^2ds\right)  \notag \\
	&=& A_{t}^{(m)}+B_{t}^{(n,m)}+C_{t}^{(n,m)}+D_{t}^{(n)},  \notag
\end{eqnarray}
where
\begin{eqnarray}
I_{n}(i)&=&\left\{j:\frac{j}{m}\in\left(\frac{i-1}{n},\frac{i}{n}\right],
\ \ 1\leqslant i\leqslant [nt]. \right\}  \notag
\end{eqnarray}

The idea of the proof is that we first let $m\rightarrow \infty$ and then $n\rightarrow \infty$, and we show that each of the terms
$A_t^{(m)}, B_t^{(n,m)}, C_t^{(n,m)}$, and $D_t^{(n)}$ converges to zero almost surely, and uniformly in $t$.
	
Let us begin with the term $C_t^{(n,m)}$. We have
\begin{eqnarray}
	\parallel C^{(n,m)}\parallel_{\infty} &\leqslant&\sum_{i=1}^{[nT]}\left|u_{%
		\frac{i-1}{n}}\right|^{2}\left|m^{2H-1}\sum_{j\in I_{n}(i)}\left| G^{H}_{%
		\frac{j}{m}}-G^{H}_{\frac{j-1}{m}}\right|^2-n^{-1}\right|.
	\notag
\end{eqnarray}

As we first let $m\rightarrow \infty$, it suffices to show that, for a fixed $n$, we have
\begin{eqnarray}
\left|m^{2H-1}\sum_{j\in I_{n}(i)}\left| G^{H}_{\frac{j}{m}}-G^{H}_{\frac{j-1}{m}}\right|^2-n^{-1}\right| \rightarrow 0. \notag
\end{eqnarray}

By assumption (A1), Lemma \ref{lma:rate} and Theorem \ref{theorem:BE_bound_QV}, we only need to prove
$$\lim_{m\to\infty}m^{2H-1}\sum_{j\in I_{n}(i)}\left| G^{H}_{\frac{j}{m}}-G^{H}_{\frac{j-1}{m}}\right|^2=n^{-1}$$
which follows from assumption (A3).

Consider next the term $A_t^{(m)}$. We have
\begin{eqnarray}
|A_{t}^{(m)}|&\leqslant&m^{2H-1}\sum_{j=1}^{[mt]}\left|\left|%
\int_{(j-1)/m}^{j/m}u_{s}dG^{H}_{s}\right|^{2}-\left|u_{\frac{j-1}{m}%
}(G^{H}_{\frac{j}{m}}- G^{H}_{\frac{j-1}{m}})\right|^2\right|.
\notag
\end{eqnarray}
We will use the following inequality, valid for any $x,y\in\mathbb{R}$,
\begin{eqnarray}  \label{eq:inequality}
\left||x|^2-|y|^2\right| &\leqslant& 2\left[|x-y|^2+|y||x-y|\right].
\end{eqnarray}

This implies
\begin{eqnarray}
|A_{t}^{(m)}|
&\leqslant&2m^{-1+2H}\sum_{j=1}^{[mt]}\left|%
\int_{(j-1)/m}^{j/m}u_{s}dG^{H}_{s}-u_{\frac{j-1}{m}}(G^{H}_{\frac{j}{m}%
}-S^{H}_{\frac{j-1}{m}})\right|^2  \notag \\
&&+2m^{-1+2H}\sum_{j=1}^{[mt]}\left|u_{\frac{j-1}{m}}\left(G^{H}_{\frac{j%
}{m}}- G^{H}_{\frac{j-1}{m}}\right)\right|\left|%
\int_{(j-1)/m}^{j/m}u_{s}dG^{H}_{s}-u_{\frac{j-1}{m}}\left(G^{H}_{\frac{j%
}{m}}-G^{H}_{\frac{j-1}{m}}\right)\right|  \notag \\
&=:&E_{(m)}(t)+R_{(m)}(t),  \notag
\end{eqnarray}
where
\begin{eqnarray}
E_{(m)}(t)&=&2m^{-1+2H}\sum_{j=1}^{[mt]}\left|%
\int_{(j-1)/m}^{j/m}u_{s}dG^{H}_{s}-u_{\frac{j-1}{m}}\left(G^{H}_{\frac{j%
}{m}}- G^{H}_{\frac{j-1}{m}}\right)\right|^2,  \notag \\
R_{(m)}(t)&=&2m^{-1+2H}\sum_{j=1}^{[mt]}\left|u_{\frac{j-1}{m}%
}\left(G^{H}_{\frac{j}{m}}- G^{H}_{\frac{j-1}{m}}\right)\right|\left|%
\int_{(j-1)/m}^{j/m}u_{s}dG^{H}_{s}-u_{\frac{j-1}{m}}\left(G^{H}_{\frac{j%
}{m}}-G^{H}_{\frac{j-1}{m}}\right)\right|.  \notag
\end{eqnarray}
For the term $E_{(m)}(t)$ we observe, by applying Young inequality \eqref{Young inequality}, that
\begin{eqnarray}
|E_{(m)}(t)| &\leqslant&c_{H,\beta,\varepsilon} m^{2H-1}
\sum_{j=1}^{[mT]}\left|var_{\frac{1}{\beta}}(u;\mathcal{I}%
_{m}(j))var_{1/(H-\varepsilon)}(G^{H};\mathcal{I}_{m}(j))\right|^2,
\notag
\end{eqnarray}
where $0<\varepsilon<H$, the constant $c_{H,\beta,\varepsilon}$ comes from inequality \eqref{Young inequality} and depends only on $H,\beta$ and $\varepsilon$, and $\mathcal{I}_{m}(j)=\left(\frac{j-1}{m},\frac{j}{m}\right]$.

 By \eqref{eq:holder_var} we have
\begin{eqnarray}
var_{\frac{1}{\beta}}(u,\mathcal{I}_{m}(j)) &\leqslant&
m^{-\beta}\|u\|_{\beta}  \notag
\end{eqnarray}
and
\begin{eqnarray}
var_{1/(H-\varepsilon)}(G^{H},\mathcal{I}_{m}(j)) &\leqslant&
m^{-(H-\varepsilon)}\|G^{H}\|_{H-\varepsilon}.  \notag
\end{eqnarray}

Thus
\begin{eqnarray}
\|E_{(m)}\|_{\infty}
&\leqslant&c_{H,\beta,\varepsilon}m^{2H-1-2\beta}
\|u\|^2_{\beta} \sum_{j=1}^{[mT]}\left|var_{1/(H-\varepsilon)}(G^{H};%
\mathcal{I}_{m}(j))\right|^2,  \notag \\
&\leqslant&Tc_{H,\beta,\varepsilon}m^{2H-1-2\beta-2(H-\varepsilon)+1}
\|u\|^2_{\beta}\|G^{H}\|^2_{(H-\varepsilon)}  \notag \\
&\leqslant&Tc_{H,\beta,\varepsilon}m^{2(\varepsilon-\beta)}
\|u\|^2_{\beta}\|G^{H}\|^2_{(H-\varepsilon)}.  \notag
\end{eqnarray}

As we can choose $\varepsilon< \beta$, this implies that $\lim_{m\rightarrow\infty}\|E_{(m)}\|_{\infty}=0$ almost surely. Similarly, we can apply
\eqref{Young inequality} to the term $R_{(m)}(t)$ to get
\begin{eqnarray}
|R_{(m)}(t)|&\leqslant&c_{H,\beta,\varepsilon}m^{-1+2H}\sum_{j=1}^{[mT]}%
\left|u_{\frac{j-1}{m}}(G^{H}_{\frac{j}{m}}-G^{H}_{\frac{j-1}{m}%
})\right|\left|var_{\frac{1}{\beta}}(u,\mathcal{I}_{m}(j))var_{1/(H-%
\varepsilon)}(G^{H},\mathcal{I}_{m}(j))\right|  \notag \\
&\leqslant&2c_{H,\beta,\varepsilon}m^{-1+2H-\beta-(H-\varepsilon)}%
\parallel u\parallel_{\beta}\parallel
G^{H}\parallel_{H-\varepsilon}\sum_{j=1}^{[mT]}\left|u_{\frac{j-1}{m}%
}(G^{H}_\frac{j}{m}-G^{H}_{\frac{j-1}{m}})\right|  \notag \\
&\leqslant&2c_{H,\beta,\varepsilon}m^{-1+H-\beta+\varepsilon}\parallel
u\parallel_{\beta}\parallel
G^{H}\parallel_{H-\varepsilon}\parallel
u\parallel_{\infty}\sum_{j=1}^{[mT]}\left|var_{1/(H-\varepsilon)}(G^{H},%
\mathcal{I}_{m}(j))\right|  \notag \\
&\leqslant&Tc_{H,\beta,\varepsilon}\parallel
u\parallel_{\beta}\parallel
G^{H}\parallel^2_{H-\varepsilon}\parallel
u\parallel_{\infty}m^{-\beta+2\varepsilon}.  \notag
\end{eqnarray}
Hence, for $\varepsilon < \frac{\beta}{2}$, we get $\parallel R_{(m)}\parallel_{\infty} \rightarrow 0$ almost surely, and consequently,
$\parallel A^{(m)}\parallel_{\infty} \rightarrow 0$ almost surely as $m\rightarrow\infty$.

It remains to study the terms $D^{(n)}_t$ and $B^{(n,m)}_t$. For the term $D^{(n)}_t$ we first observe that
for any $s\in \left[\frac{i-1}{n},\frac{i}{n}\right]$, we have
\begin{equation*}
||u_{\frac{i-1}{n}}|^2-|u_s|^2| \leqslant
2\|u\|_{\infty}\|u\|_{\beta} n^{-\beta}.
\end{equation*}

Thus we can estimate
\begin{eqnarray}
|D_{t}^{(n)}|&=&\left|n^{-1}\sum_{i=1}^{[nt]}|u_{\frac{i-1}{n}%
}|^2-\int_{0}^{t}|u_s|^2ds\right|  \notag \\
&=&\left|\sum_{i=1}^{[nt]}\int_{(i-1)/n}^{i/n}(|u_{\frac{i-1}{n}%
}|^2-|u_s|^2)ds + \int_{[nt]/n}^t|u_s|^2ds \right|  \notag \\
&\leqslant&\sum_{i=1}^{[nt]}\int_{(i-1)/n}^{i/n}\left||u_{\frac{i-1}{n}%
}|^2-|u_{s}|^2\right|ds + \int_{[nt]/n}^t|u_s|^2ds  \notag \\
&\leqslant&2T\|u\|_{\infty}\|u\|_{\beta} n^{-\beta} +
\|u\|_{\infty}|t-[nt]/n|  \notag \\
&\leqslant&2T\|u\|_{\infty}\|u\|_{\beta} n^{-\beta} +
\|u\|_{\infty}n^{-1}. \notag
\end{eqnarray}
This implies that also $\parallel D^{(n)}\parallel_{\infty}\rightarrow 0$ almost surely as $n\rightarrow 0$. It remains to study the term
$B_{t}^{(n,m)}$. First note that, by the definition of $I_n(i)$, we have
\begin{equation*}
\sum_{j=1}^{[mt]}\left|u_{\frac{j-1}{m}}\left(G^{H}_{\frac{j}{m}}-G^{H}_{%
\frac{j-1}{m}}\right)\right|^2 = \sum_{i=1}^{[nt]}\sum_{j\in I_n(i)}\left|u_{%
\frac{j-1}{m}}\left(G^{H}_{\frac{j}{m}}-G^{H}_{\frac{j-1}{m}%
}\right)\right|^2.
\end{equation*}
Together with the fact that
\begin{equation*}
|u_{\frac{j-1}{m}} - u_{\frac{i-1}{n}}|^2 \leqslant 4\parallel
u\parallel_{\infty}\parallel u\parallel_\beta n^{-\beta}
\end{equation*}
as $\frac{j}{m} \in \left(\frac{i-1}{n},\frac{i}{n}\right]$, this gives us
\begin{eqnarray}
|B_{t}^{(n,m)}|&=&\left|m^{2H-1}\left(\sum_{j=1}^{[mt]}\left|u_{\frac{j-1}{m%
}}\left(G^{H}_{\frac{j}{m}}-G^{H}_{\frac{j-1}{m}}\right)\right|^2-%
\sum_{i=1}^{[nt]}\left|u_{\frac{i-1}{n}}\right|^2\sum_{j\in
I_{n}(i)}\left|G^{H}_{\frac{j}{m}}-G^{H}_{\frac{j-1}{m}}\right|^2\right)
\right|  \notag \\
&\leqslant &m^{2H-1}\sum_{i=1}^{[nt]}\sum_{j\in I_n(i)}|u_{\frac{j-1}{m}} - u_{\frac{%
i-1}{n}}|^2\left|G^{H}_{\frac{j}{m}}-G^{H}_{\frac{j-1}{m}}\right|^2
\notag \\
&\leqslant & 4m^{2H-1}\parallel u\parallel_{\infty}\parallel
u\parallel_\beta
n^{-\beta}\sum_{i=1}^{[nt]}\sum_{j\in I_n(i)}\left|G^{H}_{\frac{j}{m}%
}-G^{H}_{\frac{j-1}{m}}\right|^2.  \notag
\end{eqnarray}
Here
\begin{equation*}
m^{2H-1}\sum_{j\in I_n(i)}\left|G^{H}_{\frac{j}{m}}-G^{H}_{\frac{j-1}{m}}\right|^2 \rightarrow n^{-1}
\end{equation*}
almost surely, and thus
\begin{equation*}
m^{2H-1}\sum_{i=1}^{[nt]}\sum_{j\in I_n(i)}\left|G^{H}_{\frac{j}{m}}-G^{H}_{\frac{j-1}{m}}\right|^2 \rightarrow t
\end{equation*}
almost surely. This implies that $\parallel B^{(n,m)}\parallel_{\infty}\rightarrow 0$ which completes the proof.
\end{proof}

For each $t\geq 0$ we denote by $\mathcal{F}_{t}^H$ the $\sigma$-field generated by the random variables $\{G_{s}^{H}, 0 \leq s\leq t\}$ and the null sets.

\begin{theorem}\label{theorem:distribution asymptotic}
Under the assumptions (A1)--(A4), we further suppose that $u=\{u_{t},t\in[0,T]\}$ is an H\"older continuous stochastic
process of order $\beta$ with $\beta >\max\left(1-H,\frac12\right)$, and measurable with respect to $\mathcal{F}_T^{H}$. Set
\begin{eqnarray}
Z_{t} &=& \int_{0}^{t}u_{s}dG^{H}_{s}.
\end{eqnarray}

Then, as $n$ tends to infinity,
\begin{eqnarray*}
n^{2H-1/2}V_{n}(Z)_{t}-\sqrt{n}\int_{0}^{t}|u_{s}|^2ds &\overset{\mathcal{L}}{\rightarrow}& v_1\int_{0}^{t}|u_{s}|^2dW_{s}
\end{eqnarray*}
$\mathcal{F}_T^H$-stably in the space $\mathcal{D}([0,T]^2)$, where $W=\{W_{t},t\in[0,T]\}$ is a Brownian motion independent
of $\mathcal{F}_{T}^H $, $v_1$ is given in Lemma \ref{lem-fdd}.
\end{theorem}

\begin{proof}
As in the proof of Theorem \ref{theorem:Consistency}, we make a decomposition
\begin{eqnarray}
n^{2H-1/2}V_{n}(Z)_{t}-\sqrt{n}\int_{0}^{t}|u_{s}|^2 ds
&=:& A_{t}^{(n)}+B_{t}^{(n)}+C_{t}^{(n)},  \notag
\end{eqnarray}
where
\begin{eqnarray*}
A_{t}^{(n)}
&=&n^{2H-1/2}\sum_{i=1}^{[nt]}\left(\left|%
\int_{(i-1)/n}^{i/n}u_{s}dG^{H}_{s}\right|^{2}-\left|u_{\frac{i}{n}%
}(G^{H}_{\frac{i}{n}}-G^{H}_{\frac{i-1}{n}})\right|^2\right),  \notag \\
\notag \\
B_{t}^{(n)}&=&n^{2H-1/2}\sum_{i=1}^{[nt]}\left|u_{\frac{i}{n}%
}\left(G^{H}_{\frac{i}{n}}-G^{H}_{\frac{i-1}{n}}\right)\right|^2-%
\frac{1}{\sqrt{n}}\sum_{i=1}^{[nt]}\left|u_{\frac{i}{n}}\right|^2,  \notag \\
C_{t}^{(n)}&=&\left(\frac{1}{\sqrt{n}}\sum_{i=1}^{[nt]}\left|u_{\frac{%
i}{n}}\right|^2-\sqrt{n}\int_{0}^{t}|u_{s}|^2ds\right).  \notag
\end{eqnarray*}

Using $\beta > \frac12$ and treating the terms $A_t^{(n)}$ and $C_t^{(n)}$ as the terms $A_t^{(m)}$ and $D_t^{(n)}$ in the proof of
Theorem \ref{theorem:Consistency}, we obtain
\begin{equation*}
\parallel A^{(n)}\parallel_\infty +\parallel C^{(n)}\parallel_\infty
\rightarrow 0
\end{equation*}
almost surely. Consider next the term $B_{t}^{(n)}$. We set
\begin{eqnarray*}
\xi_{i,n} &=& n^{2H-1/2}\left|G^{H}_{\frac{i}{n}} -G^{H}_{\frac{(i-1)}{n}}\right|^2-\frac{1}{\sqrt{n}}  \notag
\end{eqnarray*}
so that
\begin{equation*}
B_{t}^{(n)}= \sum_{i=1}^{[nt]}|u_{i/n}|^{2}\xi_{i,n}.
\end{equation*}

In order to complete the proof, we need to verify hypotheses (H1) and (H2) of Theorem \ref{theorem: Hypotheses}.
For the first hypothesis (H1), we have
\begin{eqnarray}
g_{m}(t)&=&\sum_{i=1}^{[mt]}\left(m^{2H-1/2}\left|G^{H}_{i/m}-G^{H}_{(i-1)/m}\right|^{2}-\frac{1}{\sqrt{m}}\right).\nonumber
\end{eqnarray}
From Lemma \ref{lem-fdd}, we obtain the finite dimensional distributions. Thus, by the Theorem 3 in Corcuera, Nualart and Woerner \cite{CNW}, for the
following convergence in law for $0<H<3/4$,
\begin{eqnarray}
\left(G^{H}_{t},\sqrt{m}\sum_{i=1}^{[mt]}\left(m^{2H-1}\left|G^{H}_{i/m}-G^{H}_{(i-1)/m}\right|^{2}-\frac{1}{m}\right)\right)&\overset{\mathcal{L}}{\rightarrow}&\left(G^{H}_{t},v_1 W_{t}\right),\nonumber
\end{eqnarray}
in the space $\mathcal{D}([0,T])^{2}$  equipped with the Skorohod topology, where $W =\{W_{t},t\in[0 ,T]\}$ is a Brownian motion independent of the Gaussian process  $G^{H}$, we need to prove the tightness condition, which is the second hypothesis in Theorem \ref{theorem: Hypotheses}.

For the hypothesis (H2),  using the Lemma 4.3 and  the Proposition 4.2 in
\cite{Taqqu} and replace $\mathbb{E}[(G^H_j-G_{j-1}^H)(G^H_{j+u}-G_{j+u-1}^H)]=\rho_H(u)$ by $\rho_H(u)+\theta(j,j+u)$, we have for any $1\leqslant j <k\leqslant [nT].$
\begin{eqnarray}
\mathrm{I\kern-0.16em E}\left(\left|\sum_{i=j+1}^{k}\xi_{i,n}\right|^4%
\right) &=&\frac{1}{n^2} \mathrm{I\kern-0.16em E}\left(\left|%
\sum_{i=j+1}^{k}\left(n^{2H}\left|G^{H}_{\frac{i}{n}} -G^{H}_{\frac{(i-1)}{n}%
}\right|^2-1\right)\right|^4\right)  \notag \\
&=&\frac{1}{n^2} \mathrm{I\kern-0.16em E}\left(\left|\sum_{i=j+1}^{k}H_2(G^{H}_{i}-G^{H}_{i-1})\right|^4\right)  \notag \\
&=&\frac{1}{n^2} \left(\sum_{i=j+1}^{k}\sum_{l=j+1}^{k}\left(\mathbb{E}H_2(G^{H}_{i}-G^{H}_{i-1})H_2(G^{H}_{l}-G^{H}_{l-1})\right)^2\right)^2  \notag \\
&\leq&\frac{1}{n^2} \left(\sum_{i=j+1}^{k}\sum_{l=j+1}^{k}\left(c_0\rho_H(|i-l|)+c_1\theta(i,l)\right)^2\right)^2  \notag \\
&\leq&\frac{c(k-j)^2}{n^2}\left(\sum_{i=0}^{\infty}\rho_{H}^2(i)\right)^2+\frac{c}{n^2}\left(\sum_{i=j+1}^k\sum_{l}\theta(i,l)^2\right)^2  \notag \\
&\leqslant&c\left(\frac{k-j}{n}\right)^2  \notag
\end{eqnarray}
where we use $|\theta(i,l)|^2=o(1/l)$ as $l\to\infty$ in (A4), which is convergent in summation, in the last inequality.
This concludes the proof of the theorem.
\end{proof}

\section{Application to the estimation of the integrated volatility}

\label{sec:vol} In this section we apply our main results to the estimation of the integrated volatility
$\int_{0}^{t}|\sigma_s|^{2}ds$. We consider a generalized Gaussian Ornstein-Uhlenbeck process  defined as the solution
to the stochastic differential equation
\begin{eqnarray}  \label{eq:SDE}
dX_{t} &=& -\theta X_{t}dt+\sigma_{t}dG^{H}_{t} ,
\end{eqnarray}
with some initial condition $X_0 \in \mathbb{R}$. We define the estimator $QV_{n}(X)_{t}$ for the integrated volatility
$\int_{0}^{t}|\sigma_s|^{2}ds$ as
\begin{eqnarray}  \label{eq:estimator}
QV_{n}(X)_{t} &=&n^{2H-1}V_{n}(X)_{t}, \ \ t \in [0,T].
\end{eqnarray}
We begin with two simple propositions which allows us to introduce drift to the process defined by \eqref{eq:process_Z}, which can be obtained directly from Bajja, Es-Sebaiy and Viitasaari \cite{BEV}, so we omit the detailed proof here.

\begin{proposition}\label{proposition:consistency}
Suppose that the assumptions of Theorem \ref{theorem:Consistency} prevail, and let $Y=\{Y_{t},t\in [0,T]\}$ be
a stochastic process such that, as $n$ tends to infinity,
\begin{eqnarray}
n^{2H-1}V_{n}(Y)_{t} &\rightarrow& 0  \notag
\end{eqnarray}
almost surely and uniformly in $t$. Then
\begin{eqnarray}
n^{2H-1}V_{n}(Y+Z)_{t}&\underset{n\rightarrow\infty}{\longrightarrow}%
&\int_{0}^{t}|u_s|^2ds.  \notag
\end{eqnarray}
almost surely and uniformly in $t$.
\end{proposition}

Similarly, we obtain the following result on the weak convergence.

\begin{proposition}
\label{proposition:clt} Suppose that the assumptions of Theorem \ref{theorem:distribution asymptotic} prevail,
and let $Y=\{Y_{t},t\in[0,T]\}$ be a stochastic process such that, as $n$ tends to infinity,
\begin{eqnarray}
n^{2H-1}V_{n}(Y)_{t} &\rightarrow& 0  \notag
\end{eqnarray}
and uniformly in probability. Then
\begin{eqnarray}
n^{2H-\frac{1}{2}}(Y+Z)_{t}-\sqrt{n}\int_{0}^{t}|u_s|^2ds&\underset{%
n\rightarrow\infty}{\overset{Law}{\longrightarrow}}&v_1\int_{0}^{t}|u_s|^2dW_{s}
\notag
\end{eqnarray}
$\mathcal{F}_T^H$-stably in $\mathcal{D}([0,T])$, where $W=\{W_{t},t\in [0,T]\}$ is a Brownian motion independent of $\mathcal{F}_T^H$.
\end{proposition}

Consider now the estimator \eqref{eq:estimator} for the integrated volatility. With the help of Proposition
\ref{proposition:consistency} and Proposition \ref{proposition:clt} we obtain the following results.

\begin{theorem}
Suppose that $\sigma_s$ is a H\"older continuous function of order $\beta > 1-H$. Then
\begin{equation*}
QV_{n}(X)_{t} \longrightarrow \int_{0}^{t}|\sigma_s|^2ds
\end{equation*}
almost surely and uniformly in $t$.
\end{theorem}

\begin{proof}
Recall that $X$ satisfies \eqref{eq:SDE}. Thus we have
\begin{equation*}
X_{t}= X_{0}+Y_{t}+\int_{0}^{t}\sigma_{s}dG^{H}_{s},
\end{equation*}
where $Y_{t}=-\theta\int_{0}^{t}X_{s}ds.$ It is straightforward to check that the solution $X$ is bounded on every compact interval.
Consequently, the process $Y_t$ is differentiable with bounded derivative, and thus
\begin{equation*}
V_n(Y) \leqslant \theta \parallel X\parallel_\infty^2 n^{-1}.
\end{equation*}
Now the result follows from Proposition \ref{proposition:consistency} and Theorem \ref{theorem:Consistency}.
\end{proof}

\begin{theorem}
Suppose that $\sigma = \{\sigma_{t},t\in[0,T]\}$ is H\"older continuous of order $\beta > \max\left(1-H,\frac12\right)$, and
measurable with respect to $\mathcal{F}_T^{G^{H}}$. Suppose further that $0<H<3/4$. Then
\begin{eqnarray}
\sqrt{n}\left(QV_{n}(X)_{t}-\int_{0}^{t}|\sigma_s|^2ds\right)&\underset{%
n\rightarrow\infty}{\overset{Law}{\longrightarrow}}&\int_{0}^{t}|%
\sigma_s|^2dW_{s},  \notag
\end{eqnarray}
$\mathcal{F}_T^H$-stably in the space $\mathcal{D}([0,T]^2)$, where $W=\{W_{t},t\in[0,T]\}$ is a Brownian motion independent
of $\mathcal{F}_{T}^H $.
\end{theorem}

\begin{proof}
Observing that since $0< H< 3/4$, we have, for $Y_{t}=-\theta\int_{0}^{t}X_{s}ds,$ that
\begin{equation*}
n^{2H-\frac12}V_n(Y) \leqslant \theta \parallel
X\parallel_\infty^2 n^{2H-\frac32} \rightarrow 0.
\end{equation*}
Thus the result follows directly from Proposition \ref{proposition:clt} and Theorem \ref{theorem:distribution asymptotic}.
\end{proof}


\end{document}